\documentclass{article}%
\usepackage{amssymb}
\usepackage{amsfonts}
\usepackage{amsmath}
\usepackage{graphicx}
\usepackage{amsthm}%
\newtheorem{thm}{Theorem}
\newtheorem{lm}{Lemma}
\newtheorem{vr}{Variant}

\newtheorem{remark}[lm]{Remark}

\input{Macros.tex}
\begin{document}

\author{Jin-ichi Itoh
\and J\"{o}el Rouyer
\and Costin V\^{\i}lcu}
\title{On the Theorem of the Three Perpendiculars}
\maketitle

\begin{abstract}
We show that the theorem of the three perpendiculars holds in any
$n$-dimensional space form.

\end{abstract}
\date{}

\addtocounter{thm}{-1} \renewcommand{\thevr}{\Roman{vr}}

\section{Introduction}

In geometry text books, the following theorem is usually known as the theorem
of the three perpendiculars.

\begin{thm}
\label{TZ}(In the $3$-dimensional Euclidean space) Assume the point $x$ is
outside the plane $\Pi$ and the line $\Lambda$ is included in $\Pi$; if $xy$
is orthogonal to $\Pi$, with $y\in\Pi\setminus\Lambda$, and $yz$ is orthogonal
to $\Lambda$, with $z\in\Lambda$, then $xz$ is orthogonal to $\Lambda$.
\end{thm}

%

\begin{figure}
[ptb]
\begin{center}
\includegraphics[
height=2.4327in,
width=3.2344in
]%
{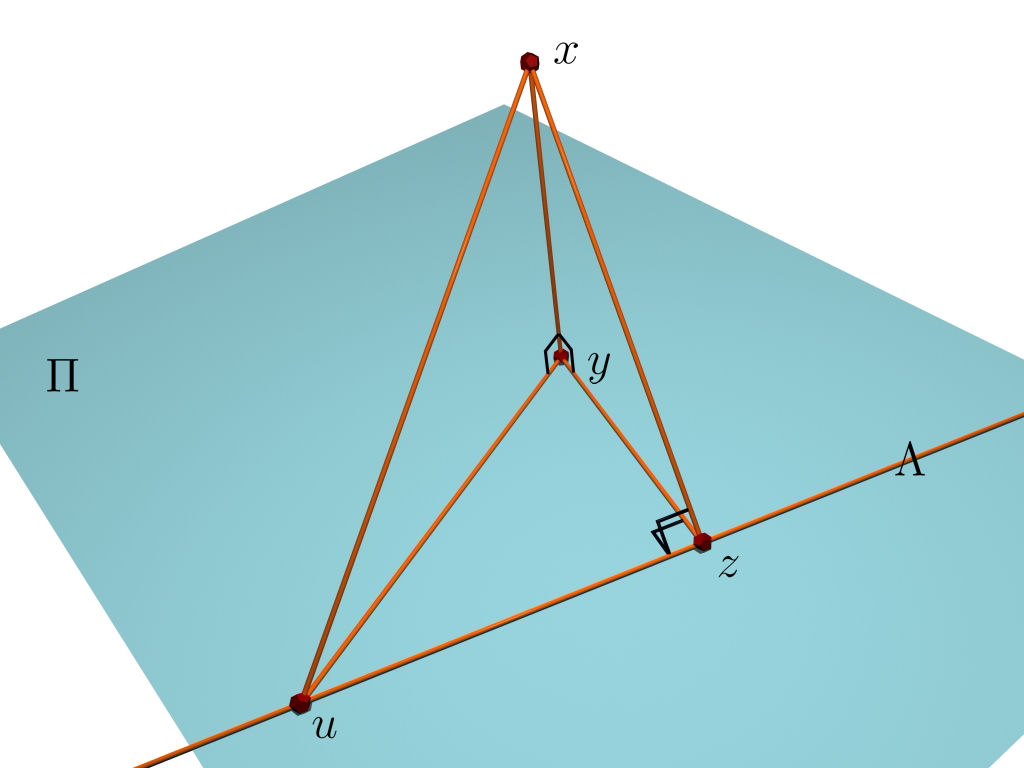}%
\caption{Equivalence between Theorem \ref{TZ} and Variant \ref{VG}.}%
\label{F3}%
\end{center}
\end{figure}

In a very short note \cite{Gupta}, H. N. Gupta observed (see Figure \ref{F3})
that Theorem \ref{TZ} is equivalent to

\begin{vr}
\label{VG}(In the $n$-dimensional Euclidean space) Let $x$, $y$, $z$, $u$ be
four pairwise distinct points. If the triangles $xyz$, $xyu$, $yzu$ are right
at $y$, $y$ and $z$ respectively, then the triangle $xzu$ is right at $z$.
\end{vr}

For the completeness of this note, we give a short classical proof of this fact.

\begin{proof}
Let $v$ be the symmetric point of $u$ with respect to $z$. Since $uzy$ and
$vzy$ are right at $z$ and $d\left(  u,z\right)  =d\left(  v,z\right)  $, the
triangles $uzy$ and $vzy$ are congruent and $d\left(  u,y\right)  =d\left(
v,y\right)  $. Since $xy$ is orthogonal to the $2$-space through $y$, $z$ and
$u$, the triangle $xyv$ is right at $y$. It follows that the triangles $xyv$
and $xyu$ are congruent and $d\left(  x,v\right)  =d\left(  x,u\right)  $.
Hence the triangles $xzv$ and $xzu$ are congruent, whence $\measuredangle
xzv=\measuredangle xzu$. Moreover $\measuredangle xzv+\measuredangle xzu=\pi$,
whence $\measuredangle xzv=\measuredangle xzu=\pi/2$.
\end{proof}

H. N. Gupta noted that Variant \ref{VG} does not emphasize any plane $\Pi$ or
line $\Lambda$, and so implies that Theorem \ref{TZ} holds for spaces of any
dimensions. He noticed that the proof of Variant \ref{VG} holds as well for
hyperbolic spaces, but he mentioned nothing about the spherical case. Instead,
he focused on the fact that the triangles mentioned in Variant \ref{VG} can be
exchanged: \emph{whenever three of the four mentioned triangles are right, the
fourth one is also right}. This last statement does not hold for spheres; see
Remark \ref{R1}.

In order to state Theorem \ref{TZ} in other spaces, one has to replace the
Euclidean lines by geodesics and the Euclidean planes with totally geodesic
surfaces. Consequently, in the statement of Variant \ref{VG}, Euclidean
triangles have to be replaced by geodesic triangles.

It is straightforward to see that the above proof of Variant \ref{VG} remains
valid in any simply connected space forms. In this note we give two new proofs
for Theorem \ref{TZ} following two distinct and more descriptive approaches;
it is just a pretext to play with standard models of space forms and their
subspaces. Moreover, the second proof is valid for any space form. We are
convinced that other methods of proof can be imagined, for example by the use
of isometries. We believe that Theorem \ref{TZ} remains valid in other spaces;
it may be interesting to find them.

\section{Preliminaries}

In a space form $\Omega$, by definition, an $n$-\emph{space} is an
$n$-dimensional totally geodesic complete submanifold. A \emph{geodesic }is
always supposed to be maximal, or, in other words, to be a $1$-space. A
geodesic through $x$ and $y$ is denoted by $xy$.

With this definition, we can give the following statement, more general than
Variant \ref{VG}.

\begin{vr}
\label{VN}(In any $n$-dimensional simply connected space form). Assume the
point $x$ is outside the $p$-space $\Pi$ and the $q$-space $\Lambda$ is
included in $\Pi$ ($n>p>q>0$). If $xy$ is orthogonal to $\Pi$, with $y\in
\Pi\setminus\Lambda$, and $yz$ is orthogonal to $\Lambda$, with $z\in\Lambda$,
then $xz$ is also orthogonal to $\Lambda$.
\end{vr}

\begin{proof}
Notice first that the proof of Variant \ref{VG} holds in any simply connected
space form. A geodesic $xz$ is orthogonal to a $q$-space $\Lambda\ni z$ if and
only if all triangles $xzu$, $u\in\Lambda$ are right at $z$. This fact and
Variant \ref{VG} yield the conclusion.
\end{proof}

For the next proofs, we need explicit models. For hyperbolic spaces, there
exist several standard models; we chose the one which is formally similar to
the standard model of spheres.

If $\Omega$ is a sphere of dimension $n$, then it is supposed to be the unit
sphere in the space $\mathbb{R}^{n+1}$ endowed with the canonical inner
product $\left\langle ~~,~~\right\rangle $.

Similarly, if $\Omega$ is the $n$-dimensional hyperbolic space, it is supposed
to be embedded in $\mathbb{R}^{1,n}$ as%
\[
\set(:\left(  x_{0},\ldots,x_{n}\right)  \in\mathbb{R}^{1,n}|\left\langle
x,x\right\rangle =1,x_{0}>0:)\text{.}%
\]
Recall that $\mathbb{R}^{1,n}$ is the space $\mathbb{R}^{n+1}$ endowed with
the pseudo-Euclidean inner product%
\[
\left\langle x,y\right\rangle =x_{0}y_{0}-x_{1}y_{1}-\ldots-x_{n}y_{n}\text{,}%
\]
where $x_{i}$ denotes the $i^{\text{th}}$ coordinate of $x$, the index $i$
starting from $0$.

For the uniformity of the presentation, the same bracket notation stands for
two distinct products, according to the case.

We will refer to $\mathbb{R}^{n+1}$ or $\mathbb{R}^{1,n}$ as the \emph{ambient
space}. For any subset $P$ of $\Omega$, $\mathrm{Sp}\left(  P\right)  $ stands
for the linear subspace of the ambient space spanned by $P$; in particular,
$\mathrm{Sp}\left(  \Omega\right)  $ is the whole ambient space.

The formula expressing the distance between two points $x$ and $y$ of $\Omega$
is:
\begin{align*}
d\left(  x,y\right)   &  =\arccos\left<  x,y\right>  ~\text{if }%
\Omega~\text{is a sphere,}~\text{and}\\
d\left(  x,y\right)   &  =\mathrm{\mathrm{arcosh}}\left<  x,y\right>
~\text{if }\Omega~\text{is a hyperbolic space.}%
\end{align*}

Another parallelism is the description of $p$-spaces: $\Phi\subset\Omega$ is a
$p$-space if and only if there exists a linear subspace $F\subset
\mathrm{Sp}\left(  \Omega\right)  $ such that $\Phi=F\cap\Omega$. Of course,
$F=\mathrm{Sp}\left(  \Phi\right)  $.


\section{The projection approach\label{SP}}

This approach begins with the observation that Theorem \ref{TZ} is an obvious
corollary of its following variant (choose $E=\mathbb{R}^{3}$, $F=\Pi$,
$G=\Lambda$)

\begin{vr}
\label{VP}Let $\left(  E,\left\langle \cdot,\cdot\right\rangle \right)  $ be a
pseudo-Euclidean vector space. Let $F$ be a proper subspace of $E$, and $G$ a
proper subspace of $F$. Assume that the restrictions of the inner product to
$E\times E$ and $F\times F$ are non-degenerate. Let $f$ and $g$ be the
orthogonal projections onto $F$ and $G$ respectively. Then $g=g\circ f$.
\end{vr}

\begin{proof}
Let $H\overset{\mathrm{def}}{=}G^{\perp}\cap F$. We have $M=F^{\bot}\oplus
H\oplus G$, and if $x=x_{F^{\perp}}+x_{H}+x_{G}$, with $x_{F^{\perp}}\in
F^{\bot}$, $x_{H}\in H$, $x_{G}\in G$, then $f\left(  x\right)  =x_{H}+x_{G}$,
$g\left(  x\right)  =x_{G}$.
\end{proof}

In a Euclidean space $E$, there are two equivalent ways to define the
projection $f$ onto the linear subspace $F$. The point $f\left(  x\right)  $
is at the same time the only point $y\in F$ such that the line $xy$ is
orthogonal to $F$ (\ie, the \emph{orthogonal} projection), and the unique
closest point to $x$ among the points of $F$ (\ie, the \emph{metrical}
projection). Metrical and orthogonal projections can both be defined in any
Riemannian manifold (even in more general spaces), onto any closed
submanifold. It is well-known that a metrical projection is always an
orthogonal projection; this follows from the first variation formula. It is
also clear that any point admits at least one metrical projection. In the
Euclidean case, the two notions coincide.

In the case of hyperbolic spaces, any point has at most one orthogonal
projection on any totally geodesic subspace. This is an obvious consequence of
the fact that the sum of the angles of any triangle is less than $\pi$.
Therefore, the metrical and orthogonal projections coincide and are
single-valued, as in the Euclidean case.

In the case of spherical spaces, however, the situation is slightly more complicated.

\begin{lm}
\label{LPrH}Let $\Phi$ be a $p$-space of $\Omega\simeq\mathbb{\mathbb{S}}^{n}$
($0<p<n$) and $x\in\Omega\setminus\Phi$. Then $y\in\Phi$ is an orthogonal
projection of $x$ onto $\Phi$ if and only if either $y$ or $-y$ is a metrical
projection of $x$ onto $\Phi$.
\end{lm}

\begin{proof}
Assume $y\in\Phi$. The point $y$ is an orthogonal projection of $x$ if and
only if $\mathrm{Sp}\left(  xy\right)  =\mathbb{R}y\oplus U$, $\mathrm{Sp}%
\left(  \Phi\right)  =\mathbb{R}y\oplus V$, with $y\perp U$, $y\perp V$ and
$U\perp V$. This proves that $y$ is an orthogonal projection if and only if
$-y$ is so. Assume now that $d\left(  x,y\right)  \leq\pi/2$ (interchange $x$
and $-x$ if necessary). Let $z$ be a metrical projection of $x$. Put
$a=d\left(  x,y\right)  $, $b=d\left(  x,z\right)  $ and $c=d\left(
y,z\right)  $, and notice that, by the choice of $y$, $\cos\left(  a\right)  $
and $\cos\left(  b\right)  $ are non-negative. The spherical triangle $xyz$ is
right at $y$, therefore $\cos\left(  b\right)  =\cos\left(  a\right)
\cos\left(  c\right)  \leq\cos\left(  a\right)  $. On the other hand, since
$z$ is a metrical projection, $\cos\left(  b\right)  \geq\cos\left(  a\right)
$, whence $a=b$ and $y$ is also a metrical projection.
\end{proof}

%

\begin{figure}
[ptb]
\begin{center}
\includegraphics[
height=2.4275in,
width=3.2283in
]%
{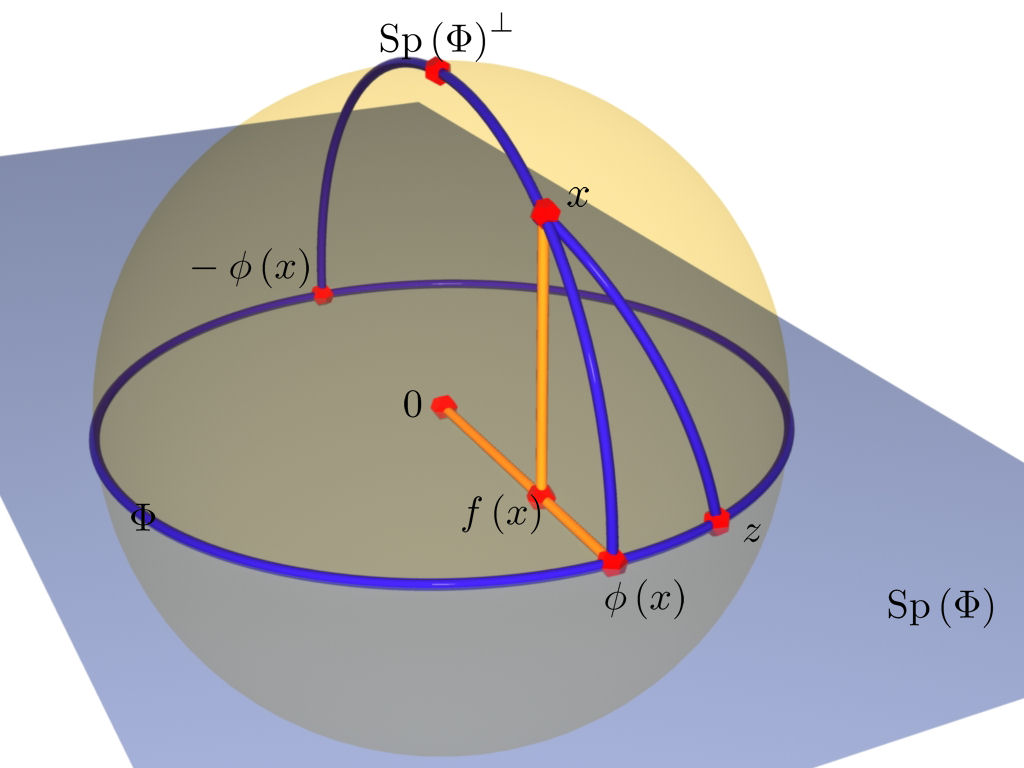}%
\caption{Metrical projection on spheres.}%
\label{F1}%
\end{center}
\end{figure}
\begin{figure}
[ptbptb]
\begin{center}
\includegraphics[
height=2.4275in,
width=3.2283in
]%
{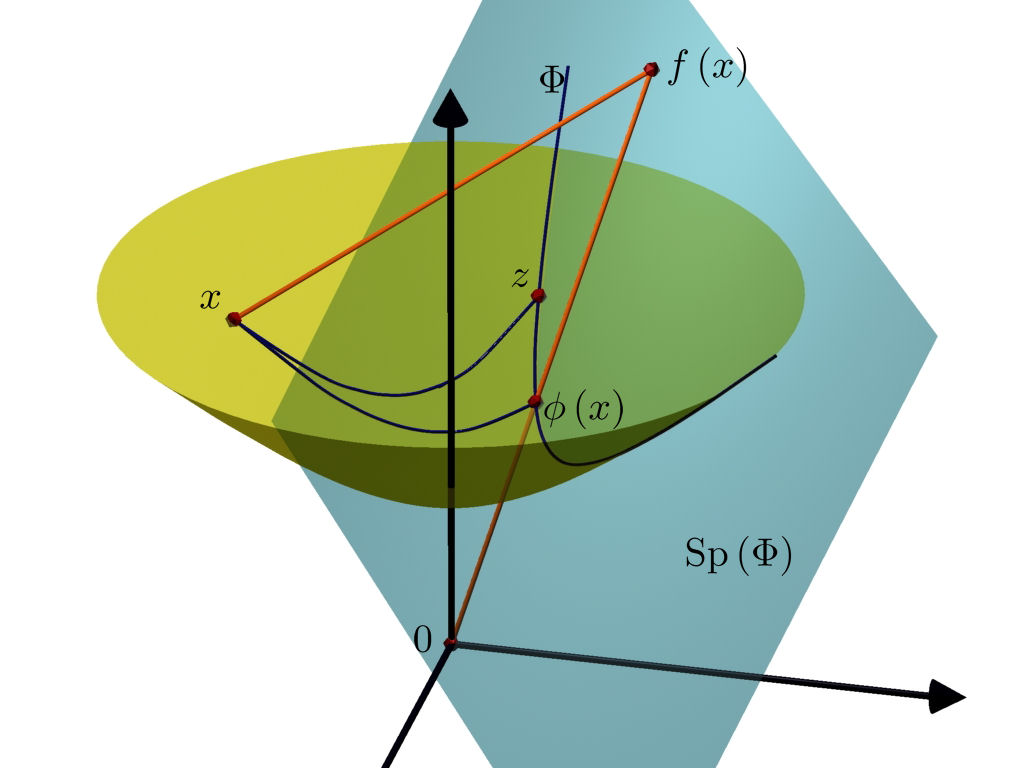}%
\caption{Metrical projection on hyperbolic spaces.}%
\label{F2}%
\end{center}
\end{figure}

\begin{lm}
\label{LPrS}Let $\Phi$ be a $p$-space of $\Omega\simeq\mathbb{\mathbb{S}}^{n}$
($0<p<n$) and $x\in\Omega$.

\begin{enumerate}
\item \label{Q1}If $x\in\mathrm{Sp}\left(  \Phi\right)  ^{\perp}$ then any
$y\in\Phi$ is a metrical projection of $x$ onto $\Phi$.

\item \label{Q2}If $x\notin\mathrm{Sp}\left(  \Phi\right)  ^{\perp}$ then $x$
has a unique metrical projection $\phi\left(  x\right)  $. The map
$\phi:\Omega\setminus\mathrm{Sp}\left(  \Phi\right)  ^{\perp}\rightarrow\Phi$
satisfies $\phi=s\circ f$, where $f:$ $\mathrm{Sp}\left(  \Omega\right)
\rightarrow\mathrm{Sp}\left(  \Phi\right)  $ is the orthogonal projection, and
$s:\mathrm{Sp}\left(  \Omega\right)  \setminus\left\{  0\right\}
\rightarrow\Omega$ is the radial projection.
\end{enumerate}

\begin{proof}
\ref{Q1}. It is clear from the fact that $d\left(  x,y\right)  =\pi/2$ for any
$y\in\Phi$.

\ref{Q2}. Let $z$ be an arbitrary point in $\Phi$; we have
\[
\cos d\left(  x,z\right)  =\left\langle x,z\right\rangle =\left\langle
f\left(  x\right)  ,z\right\rangle =\left\Vert f\left(  x\right)  \right\Vert
\cos d\left(  z,s\circ f\left(  x\right)  \right)  \text{,}%
\]
whence the unique global minimum of $d\left(  x,\cdot\right)  |\Phi$ is
$s\circ f\left(  x\right)  $ (see Figure \ref{F1}).
\end{proof}
\end{lm}

\begin{remark}
\label{R1}Let $\Phi$ be the $2$-sphere through $y$, $z$, and $u$. If $x\in
Sp\left(  \Phi\right)  ^{\perp}$ then the triangles $xyz$, $xzu$, and $xyu$
are right at $y$, $z$ and $y$ respectively, but $yzu$ is not right in general.
Therefore, as stated in the introduction, the triangles mentioned in Variant
\ref{VG} cannot be exchanged in the case of spheres.
\end{remark}

A result similar to the second part of Lemma \ref{LPrS} holds in the
hyperbolic case (see Figure \ref{F2}).

\begin{lm}
\label{LPrH2}Let $\Phi$ be a $p$-space of $\Omega\simeq\mathbb{\mathbb{H}}%
^{n}$ ($0<p<n$). Any $x\in\Omega$ admits a unique metrical projection
$\phi\left(  x\right)  $ onto $\Phi$. Moreover the map $\phi:\Omega
\rightarrow\Phi$ satisfies $\phi=s\circ f$, where $f:$ $\mathrm{Sp}\left(
\Omega\right)  \rightarrow\mathrm{Sp}\left(  \Phi\right)  $ is the
orthogonal projection, and $s:\left\{  x\in\mathrm{Sp}\left(
\Omega\right)  |\left\langle x,x\right\rangle >0,x_{0}>0\right\}
\rightarrow\Omega$ is the radial projection.
\end{lm}

\begin{proof}
The proof is similar to the proof of Lemma \ref{LPrS}, with hyperbolic cosines
instead of cosines. The only difficulty is to prove that $D\overset
{\mathrm{def}}{=}\left\{  x\in\mathrm{Sp}\left(  \Omega\right)  |\left\langle
x,x\right\rangle >0,x_{0}>0\right\}  $ is stable under $f$; \ie, $f\left(
D\right)  \subset D$. First notice that the proof reduces to the three
dimensional case. We can assume without loss of generality that $\mathrm{Sp}%
\left(  \Phi\right)  $ has equation $x_{0}=ax_{1}$ ($a>1$). By straightforward
computation, the matrix of $f$ in the canonical basis is%
\[
\frac{1}{a^{2}-1}\left(
\begin{array}
[c]{lll}%
a^{2} & -a & 0\\
a & -1 & 0\\
0 & 0 & a^{2}-1
\end{array}
\right)  \text{.}%
\]
It follows that
\[
\left\langle f\left(  x\right)  ,f\left(  x\right)  \right\rangle
=\frac{(x_{0}-ax_{1})^{2}}{a^{2}-1}+\left\langle x,x\right\rangle \text{,}%
\]
whence $\left\langle f\left(  x\right)  ,f\left(  x\right)  \right\rangle >0$
whenever $\left\langle x,x\right\rangle >0$. Moreover, $x\in D$ implies
$ax_{0}-x_{1}>0$, and consequently the first coordinate of $f\left(  x\right)
$ is positive.
\end{proof}

By Lemmas \ref{LPrH} and \ref{LPrS}, the spherical and hyperbolic variants of
Theorem \ref{TZ} are a consequence of

\begin{vr}
Let $\Phi$ be a $p$-space of $\Omega\simeq\mathbb{\mathbb{H}}^{n}$ or
$\mathbb{\mathbb{S}}^{n}$ and $\Gamma$ be a $q$-space included in $\Phi$
($0<q<p<n$). Let $\phi:D_{\phi}\rightarrow\Phi$ and $\gamma:D_{\gamma
}\rightarrow\Gamma$ be the metrical projections onto their respective images,
where $D_{\phi}=\Omega\setminus\mathrm{Sp}\left(  \Phi\right)  ^{\perp}$ and
$D_{\gamma}=\Omega\setminus\mathrm{Sp}\left(  \Gamma\right)  ^{\perp}$ if
$\Omega\simeq\mathbb{\mathbb{S}}^{n}$ and $D_{\phi}=D_{\gamma}=\Omega$
otherwise. Then $\gamma=\gamma\circ\phi$.
\end{vr}

\begin{proof}
Define $s:\left\{  x\in\mathrm{Sp}\left(  \Omega\right)  |\left\langle
x,x\right\rangle >0\right\}  \rightarrow\Omega$ by $s\left(  x\right)
=x/\sqrt{\left\langle x,x\right\rangle }$. By Lemmas \ref{LPrS} and
\ref{LPrH2}, $\phi=s\circ f$ and $\gamma=s\circ g$, where $f$ and $g$ are the
orthogonal projections onto $\mathrm{Sp}\left(  \Phi\right)  $ and
$\mathrm{Sp}\left(  \Gamma\right)  $ respectively. From the definition of $s$
and the linearity of $g$, we get $s\circ g\circ s=s\circ g$. From this fact
and Variant \ref{VP} we get
\[
\gamma\circ\phi=s\circ g\circ s\circ f=s\circ g\circ f=s\circ g=\gamma\text{.}%
\]

\end{proof}


\section{The constant angle approach}

\renewcommand{\theenumi}{\roman{enumi}}

This section is devoted to our second method of proof.

\begin{lm}
Let $\Phi$ and $\Delta$ be two distinct $2$-spaces in space form $\Omega$.

\begin{enumerate}
\item \label{P1}$\gamma\overset{\mathrm{def}}{=}\Phi\cap\Delta$ is a geodesic.

\item \label{P2}The angle between $\Phi$ and $\Delta$ is constant along
$\gamma$.
\end{enumerate}
\end{lm}

\begin{proof}
The statements are local; therefore it is sufficient to prove them for simply
connected space forms. The Euclidean case is clear, so we can assume
$\Omega\simeq\mathbb{S}^{3}$ or $\Omega\simeq\mathbb{H}^{3}$.

\ref{P1}. Put $D=\mathrm{Sp}\left(  \Delta\right)  $ and $F=\mathrm{Sp}\left(
\Phi\right)  $. Then $\gamma=\Omega\cap D\cap F$ is a $d$-space, with
$d=\dim\left(  D\cap F\right)  -1$. Since $\Phi$ and $\Delta$ are $2$-spaces,
$\dim\left(  D\right)  =\dim\left(  G\right)  =3$. By hypothesis we have
$D\neq F$, whence $D+F=\mathrm{Sp}\left(  \Omega\right)  $. Now
\[
\dim\left(  D\cap F\right)  =\dim\left(  D\right)  +\dim\left(  F\right)
-\dim\left(  D+F\right)  =2\text{.}%
\]

\ref{P2}. Let $z$ be a point of $\gamma$. Let $u$, $v\in\mathrm{Sp}\left(
\Omega\right)  $ be unit normal vectors to $\mathrm{Sp}\left(  D\right)  $ and
$\mathrm{Sp}\left(  F\right)  $ respectively. Note that, in the case
$\Omega\simeq\mathbb{H}^{3}$, $D$ and $F$ cannot be tangent to the isotropic
cone, so $\mathrm{Sp}\left(  D\right)  =u^{\perp}$ and $\mathrm{Sp}\left(
D\right)  =v^{\perp}$. Let $n_{D}\left(  z\right)  $ (resp. $n_{F}\left(
z\right)  $) be a vector of $T_{z}\Omega$ normal to $T_{z}\Delta$ (resp.
$T_{z}\Phi$). Obviously, $T_{z}\Omega=z^{\perp}$ and%
\[
T_{z}\Delta=D\cap T_{z}\Omega=u^{\perp}\cap z^{\perp}=\left(  \mathbb{R}%
u+\mathbb{R}z\right)  ^{\perp}\text{.}%
\]
Hence $n_{D}\left(  z\right)  \in\left(  \mathbb{R}u+\mathbb{R}z\right)  \cap
z^{\perp}=\mathbb{R}u$. It follows that $\measuredangle\left(  \mathbb{R}%
n_{D}\left(  z\right)  ,\mathbb{R}n_{F}\left(  z\right)  \right)
=\measuredangle\left(  \mathbb{R}u,\mathbb{R}v\right)  $ does not depend on
$z\in\gamma$.
\end{proof}

If the constant angle between two $2$-spaces of a $3$-dimensional space form
is $\pi/2$, then the two spaces are said to be orthogonal.

\begin{proof}
[Second proof of Theorem \ref{TZ} for space forms.]Denote by $\Omega$ the
$3$-dimensional space form where everything takes place. Let $\Gamma$ be the
$2$-space containing $x$, $y$ and $z$. Let $u\in T_{z}\Omega$ be a vector
normal to $\Gamma$. Since $\Gamma\supset xy\perp\Pi$, $\Gamma$ and $\Pi$ are
orthogonal, whence $T_{z}\Lambda=\left(  T_{z}\left(  yz\right)  ^{\perp}\cap
T_{z}\Pi\right)  $, $\left(  T_{z}\left(  yz\right)  ^{\perp}\cap T_{z}%
\Pi\right)  \perp\left(  T_{z}\left(  yz\right)  ^{\perp}\cap T_{z}%
\Gamma\right)  $, and $\left(  T_{z}\left(  yz\right)  ^{\perp}\cap
T_{z}\Gamma\right)  =\left(  T_{z}\left(  yz\right)  +\mathbb{R}u\right)
^{\perp}$. Hence $T_{z}\Lambda=\mathbb{R}u$, \ie, $T_{z}\Lambda\perp
T_{z}\Gamma\supset T_{z}\left(  xy\right)  $.
\end{proof}

\bigskip

\textbf{Acknowledgement.} The second and third authors were supported by the
grant PN-II-ID-PCE-2011-3-0533 of the Romanian National Authority for
Scientific Research, CNCS-UEFISCDI.

For various remarks, we thank Dorin Cheptea, Cristina Cosoreanu, Liana David,
Radu Pantilie and Tudor Zamfirescu.


\bigskip

Jin-ichi Itoh

\noindent{\small Faculty of Education, Kumamoto University \newline Kumamoto
860-8555, JAPAN \newline j-itoh@gpo.kumamoto-u.ac.jp}

\medskip

J\"oel Rouyer

\noindent{\small Institute of Mathematics ``Simion Stoilow'' of the Romanian
Academy, \newline P.O. Box 1-764, Bucharest 70700, ROMANIA \newline
Joel.Rouyer@ymail.com, Joel.Rouyer@imar.ro}

\medskip

Costin V\^{\i}lcu

\noindent{\small Institute of Mathematics ``Simion Stoilow'' of the Romanian
Academy, \newline P.O. Box 1-764, Bucharest 70700, ROMANIA \newline
Costin.Vilcu@imar.ro}

\end{document}